\documentclass[leqno,12pt]{article}
\usepackage{amssymb,amsmath,amsfonts,amsthm}
\usepackage{authblk}
\usepackage{epsfig,color,mathrsfs}
\usepackage{graphicx}
\usepackage{amscd}
\usepackage{caption}
\usepackage{mathrsfs,bbm}
\usepackage{multicol}
\usepackage{color}

\newcommand{\be}{\begin{equation}}
\newcommand{\ee}{\end{equation}}
\newcommand{\ben}{\begin{eqnarray*}}
\newcommand{\een}{\end{eqnarray*}}

\newcommand{\R}{\mathbb R}
\newcommand{\C}{\mathbb C}

\allowdisplaybreaks

\newtheorem{theorem}{Theorem}[section]

\newtheorem{corollary}[theorem]{Corollary}

\newtheorem{lemma}[theorem]{Lemma}

\newtheorem{proposition}[theorem]{Proposition}

\definecolor{darkgreen}{rgb}{0.09, 0.45, 0.27}
\definecolor{debianred}{rgb}{0.84, 0.04, 0.33}

\numberwithin{equation}{section}
\begin{document}
\title{\vskip-0.3in Steady state solutions for the Gierer-Meinhardt system in the whole space} 

\author[ ]{Marius Ghergu}

\affil[ ]{School of Mathematics and Statistics}
\affil[ ]{University College Dublin }
\affil[ ]{Belfield Campus, Dublin 4}
\affil[ ]{E-mail: {\tt marius.ghergu@ucd.ie}}
%\date{ 6 February, 2019 }

\maketitle

\begin{abstract} 
We are concerned with the study of positive solutions to the Gierer-Meinhardt system
$$
\begin{cases}
\displaystyle   -\Delta u+\lambda u=\frac{u^p}{v^q}+\rho(x) &\quad\mbox{ in }\R^N\, , N\geq 3,\\[0.1in]
\displaystyle   -\Delta v+\mu v=\frac{u^m}{v^s} &\quad\mbox{ in }\R^N,\\[0.1in]
\end{cases}
$$
which satisfy $u(x), v(x)\to 0$ as $|x|\to \infty$.
In the above system $p,q,m,s>0$, $\lambda, \mu\geq 0$ and $\rho\in C(\R^N)$, $\rho\geq 0$. It is a known fact that posed in a smooth and bounded domain of $\R^N$, the above system subject to homogeneous Neumann boundary conditions has positive solutions if $p>1$ and $\sigma=\frac{mq}{(p-1)(s+1)}>1$. In the present work we emphasize a different phenomenon:
we see that for $\lambda, \mu>0$ large,  positive solutions with exponential decay exist if $0< \sigma\leq 1$.  Further, for $\lambda=\mu=0$ we derive various existence and nonexistence results and underline the role of the critical exponents $p=\frac{N}{N-2}$ and $p=\frac{N+2}{N-2}$.
\end{abstract}

\noindent{\bf Keywords: Gierer-Meinhardt system; steady state solutions; existence and nonexistence} 

\medskip

\noindent{\bf 2020 AMS MSC: 35J47, 35B53, 35C15, 35B45} 
%\newpage

%\tableofcontents

\section{Introduction and the main results}

In 1972, Gierer and Meinhardt \cite{GM72} derived a mathematical model to account for {\it morphogenesis} in hydra, a tissues regeneration process through controlling the spatial distribution of cells. The morphogenesis was first observed by Trembley \cite{T44} in 1744 and is directed by two types of factors - chemical molecules and mechanical forces. Based on the findings of Turing \cite{T52} in 1952, the Gierer and Meinhardt model assumes the existence of two chemical substances (morphogens): a slowly diffusing  {\it activator} and a fast diffusing {\it inhibitor}. 
Their original model in \cite{GM72} reads as

$$
(GM)\quad 
\begin{cases}
\displaystyle   u_t=d\Delta u-\lambda u+\frac{u^p}{v^q}+\rho(x) &\quad\mbox{ in }\Omega\times (0, T),\\[0.1in]
\displaystyle   v_t=D\Delta v-\mu v+\frac{u^m}{v^s} &\quad\mbox{ in }\Omega\times (0, T),\\[0.1in]
\displaystyle \frac{\partial u}{\partial \nu}=\frac{\partial v}{\partial \nu}=0&\quad\mbox{ on }\partial\Omega\times (0, T),\\[0.1in]
u(x,0)=u_0(x), v(x, 0)=v_0(x)&\quad\mbox{ in }\Omega,\\[0.1in]
\end{cases}
$$
where $\Omega\subset\R^N$ is a smooth and bounded domain, $\nu$ denotes the outer unit normal vector at $\partial\Omega$, $\lambda, \mu\geq 0$ and $\rho\in C(\overline\Omega)$ is a nonnegative function. The nonnegative solution $(u,v)$ of the above system defines the concentrations of the activator and inhibitor respectively while $\rho$ is related to  the basic production rate of the activator.
Finally,  the exponents $p, q,m,s$  are assumed to satisfy
$$
p>1\,, \quad q>0\,,\quad m>0\,\quad \mbox{ and }\quad s>-1.
$$
The inhibitor diffuses much faster compared to the activator, which means that the diffusion coefficients $d$ and $D$ satisfy $D\gg d>0$. 

As explained in \cite{KS17}, the dynamics of the Gierer-Meinhardt system is  characterised by two numbers: 
\begin{itemize}
\item the net self-activation index $(p-1)/m$ which is responsible for the strength of self-activation of the activator;
\item the net cross-inhibition index $q/(s+1)$ which measures how strongly the inhibitor suppresses the production of the activator and that of itself.
\end{itemize}  
From biological interpretations of the above exponents, it is assumed that 
$$
0<\frac{p-1}{m}<\frac{q}{s+1},
$$
that is, 
$$
\sigma:=\frac{mq}{(p-1)(s+1)}>1.
$$
Under the assumption $\sigma>1$ it was obtained in \cite{J06} the existence of global in time solutions to $(GM)$; see also \cite{CSX17, DKP07,  JN07, KPW21, LPS17, NST06}. As already emphasised in \cite{NST06}, various existence and nonexistence results on the system $(GM)$ can be derived according to the case where the basic production rate $\rho$ of the activator is identically zero or not.  

One particular perspective in understanding the structure of $(GM)$ is given by the study of the {\it shadow system}, an approach introduced by Keener \cite{K78}. More precisely, letting $D\to \infty$ in $(GM)$ we deduce $v(x,t)=w(t)$, $t\in (0, T)$ and so we are lead to consider:
$$
\begin{cases}
\displaystyle   u_t=d\Delta u-\lambda u+\frac{u^p}{w^q}+\rho(x) &\quad\mbox{ in }\Omega\times (0, T),\\[0.1in]
\displaystyle   w_t=-\mu w+\frac{u^m}{w^s}   &\quad\mbox{ in }\Omega\times (0, T),\\[0.1in]
\displaystyle \frac{\partial u}{\partial \nu}=0&\quad\mbox{ on }\partial\Omega\times (0, T),\\[0.1in]
u(x,0)=u_0(x), w(0)=w_0&\quad\mbox{ in }\Omega.\\[0.1in]
\end{cases}
$$
Various results on the shadow system of $(GM)$ are obtained in  \cite{DGKZ22, KS17, M06, M07, WXZZ17}. 

\medskip

In the present work we investigate the steady state solutions of the Gierer-Meinhardt system $(GM)$ in the whole space. Precisely, we consider the system
\begin{equation}\label{GM0}
\begin{cases}
\displaystyle   -\Delta u+\lambda u=\frac{u^p}{v^q}+\rho(x) &\quad\mbox{ in }\R^N\, , N\geq 3,\\[0.1in]
\displaystyle   -\Delta v+\mu v=\frac{u^m}{v^s} &\quad\mbox{ in }\R^N,\\[0.1in]
u(x), v(x)\to 0\mbox{ as }|x|\to \infty,
\end{cases}
\end{equation}
where $\rho\in C(\R^N)$, $\rho\geq 0$, $p,q,m,s>0$ and $\lambda, \mu\geq 0$. We are looking for positive classical solutions of \eqref{GM0}, that is, functions $u, v\in C^2(\R^N)$ such that $u, v>0$ and satisfy \eqref{GM0} at every point of $\R^N$. 
In a smooth and bounded domain, the above system was discussed in \cite{G09, GR08, JN07, KW08, PSW17, WW99} under various boundary conditions. In \cite{KWY13} the following related system in $\R^3$ was investigated:
\begin{equation}\label{WW1}
\begin{cases}
\displaystyle   -\varepsilon \Delta u+ u=\frac{u^2}{v}&\quad\mbox{ in }\R^3,\\[0.1in]
\displaystyle   -\Delta v+ v=u^2 &\quad\mbox{ in }\R^3,\\[0.1in]
u(x), v(x)\to 0\mbox{ as }|x|\to \infty.
\end{cases}
\end{equation}
It is obtained in \cite{KWY13} that for $\varepsilon>0$ sufficiently small, the system \eqref{WW1} admits a positive radial solution $(u,v)$ such that $u$ and $v$ decay exponentially at infinity. Similar results in dimension $N=2$ were previously obtained  in \cite{DKW03} and \cite{NW06}.

Let us note that if $(u,v)$ is a positive solution to \eqref{WW1}, then $(U, V)=(\varepsilon^{-1} u, \varepsilon^{-2} v)$ is a solution to our system \eqref{GM0} for $\lambda=\varepsilon^{-1}$, $\mu=\varepsilon^{-2}$, $\rho\equiv 0$ and $p=2$, $q=1$, $m=2$, $s=0$. In this case $\sigma=1$.
 
Turning to our system \eqref{GM0}, we discuss the existence and nonexistence of solutions when $\rho\geq 0$.  We shall see that only the case $p>1$ is relevant to our study  as for $0<p\leq 1$ no positive solutions to \eqref{GM0} exist (see Theorem \ref{th1} below). Furthermore, the existence of solutions which decay exponentially at infinity  is closely related to the conditions $p>1$ and $\sigma\leq 1$ as for \eqref{WW1}.

Throughout this paper, for two positive functions $f, g\in C(\R^N)$ we use  the symbol $f(x)\simeq g(x)$ to denote that the quotient $f/g$ is bounded between two positive constants in a neighbourhood of infinity. Also, we say that $f\in C(\R^N)$ decays exponentially at infinity if $f(x)\simeq e^{-a|x|}$ for some $a>0$. 

Our first result on \eqref{GM0} concerns the case $\lambda, \mu>0$ and reads as follows. 

\begin{theorem}\label{th1}
Assume $\lambda, \mu>0$.

\begin{enumerate}
\item[{\rm (i)}] If $0<p\leq 1$ then \eqref{GM0} has no positive solutions.

\item[{\rm (ii)}] If $p>1$, $\sigma>1$ and $\mu >  \big(\frac{m}{s+1}\big)^2\lambda$, then the system \eqref{GM0} has no positive solutions $(u,v)$ with $u$ decaying exponentially at infinity.

\item[{\rm (iii)}] Assume $p>1\geq \sigma$, $\rho>0$ and $\rho(x)\simeq e^{-a|x|}$ for some $a>0$. Then, there exist two positive constants $C_1, C_2>0 $ depending on $p,q,m,s$ and $\rho$ such that the system \eqref{GM0} has positive solutions $(u,v)$ decaying exponentially at infinity provided that
\begin{equation}\label{lamm}
C_1\leq \mu\leq C_2 \lambda ^{\frac{p(s+1)}{q}-m}.
\end{equation}
\end{enumerate}
\end{theorem}
The constants $C_1$ and $C_2$ are described explicitly in the proof of Theorem \ref{th1} which we present in Section \ref{pt1}.
Observe that since $\sigma\leq 1$ one has
$$
\frac{p(s+1)}{q}-m=\frac{s+1}{q}+m\Big(\frac{1}{\sigma}-1\Big)>0,
$$
so the exponent of $\lambda$ in \eqref{lamm} is positive. In particular, \eqref{lamm} states that solutions of \eqref{GM0} with exponential decay exist if $\lambda, \mu$ are large enough. 

In the following we take $\lambda=\mu=0$ in \eqref{GM0} and derive a different set of result. We shall thus investigate the system
\begin{equation}\label{GM}
\begin{cases}
\displaystyle   -\Delta u=\frac{u^p}{v^q}+\rho(x)  & \mbox{ in }\R^N,\\[0.1in]
\displaystyle   -\Delta v=\frac{u^m}{v^s} & \mbox{ in }\R^N,\\[0.1in]
u(x), v(x)\to 0\mbox{ as }|x|\to \infty,
\end{cases}
\end{equation}

We first state the following nonexistence result. 
\begin{theorem}\label{thgs} The systems \eqref{GM} has no positive solutions in any of the following cases:

\begin{enumerate}
\item[{\rm (i)}] $0<p\leq \frac{N}{N-2}$ or $0<m\leq 2/(N-2)$.
\item[{\rm (ii)}] $\displaystyle \int_{\R^N} \rho(x)|x|^{2-N} dx=\infty$.
\item[{\rm (iii)}] $\displaystyle \int_{\R^N}|x|^{2-N} \Big(\int_{\R^N} \frac{\rho(y)}{|x-y|^{N-2}}dy\Big)^m dx=\infty$.

\item[{\rm (iv)}] $\frac{N}{N-2}<p<\frac{N+2}{N-2}$, $\rho\equiv 0$ and there exists $C>0$  such that 
\begin{equation}\label{convr}
|\nabla v(x)|\leq \frac{C }{|x|} v(x)\quad\mbox{ for all $|x|>1$ large.}
\end{equation}
\end{enumerate}
\end{theorem}
Part (iv) in Theorem \ref{thgs} above is a consequence of Theorem 4.1 in Gidas and Spruck \cite{GS81}. The result below  provides sufficient conditions for which \eqref{convr} holds. 

\begin{corollary}\label{cor1}
Assume $\frac{N}{N-2}<p<\frac{N+2}{N-2}$ and $\rho\equiv 0$. Then, the system \eqref{GM} has no positive solutions $(u,v)$ which satisfy one of the conditions below:
\begin{enumerate}
\item[{\rm (i)}] There exists $0<a<\frac{(N-2)s+N}{m}$ such that $u(x)\simeq |x|^{-a}$.
\item[{\rm (ii)}] $u$ and $v$ are radial. 
\end{enumerate}

\end{corollary}

We next turn to existence results for the system \eqref{GM}. Using a dynamical system approach, it is obtained in \cite{BG10} the existence of positive radial solutions to \eqref{GM} if  $\rho\equiv 0$.  In the following we shall discuss the case where $\rho>0$ satisfies 
\begin{equation}\label{roro}
\alpha (1+|x|^2)^{-a/2}\leq \rho(x) \leq \beta  (1+|x|^2)^{-a/2}\quad\mbox{ for all }x\in \R^N,
\end{equation}
for some $\beta>\alpha>0$. 

\begin{theorem}\label{thex} Assume $\rho$ satisfies \eqref{roro}.

\begin{enumerate}
\item[{\rm (i)}] If $0\leq a\leq 2\big(1+\frac{1}{m}\big)$ then \eqref{GM} has no solutions.

\item[{\rm (ii)}] If $2\big(1+\frac{1}{m}\big)< a<N$, $0<\sigma<1$ and the following conditions hold
\begin{equation}\label{ams}
m(a-2)<(N-2)s+N,
\end{equation}
\begin{equation}\label{eqaa}
\frac{2p}{p-1}\leq a+\sigma \Big\{2\Big(1+\frac{1}{m}\Big)-a\Big\},
\end{equation}
then, there exist $\varepsilon, \delta>0$ depending on $p, q,m,s$ and $a$ such that for
\begin{equation}\label{ed}
0<\alpha<\varepsilon \quad\mbox{ and }\quad \alpha<\beta<\delta \alpha^\sigma,
\end{equation}
the system \eqref{GM} has positive solutions.
\end{enumerate}
\end{theorem}
Since $a<N$, it is easy to check that condition \eqref{ams} holds if either $m\leq \frac{N}{N-2}$ or $m\leq s+1$. We also have:
\begin{corollary}\label{corex}
Assume $m\leq \min\{p-1, s+1\}$, $0<\sigma<1$ and $\rho$ satisfies \eqref{roro} and \eqref{ed} where $2\big(1+\frac{1}{m}\big)< a<N$. Then, the system \eqref{GM} has positive solutions.
\end{corollary}

In relation to Theorem \ref{thex}(ii), we note that since $2\big(1+\frac{1}{m}\big)< a<N$, condition \eqref{eqaa} yields  
$$
\frac{2p}{p-1}\leq a<N\Longrightarrow p>\frac{N}{N-2}.
$$
In fact, the exponent $\frac{N}{N-2}$ is optimal in the following sense:
\begin{corollary}\label{cor2}
\begin{enumerate}
\item[{\rm (i)}] If $0<p\leq \frac{N}{N-2}$ then the system \eqref{GM} has no positive solutions for all $q,m,s>0$ and any  function $\rho\in C(\R^N)$, $\rho\geq 0$.
\item[{\rm (ii)}] Conversely, for all $p>\frac{N}{N-2}$ there exist $q,m, s>0$ and a continuous function $\rho\in C(\R^N)$, $\rho\geq 0$, such that \eqref{GM} has a positive solution.
\end{enumerate}
\end{corollary}
Similarly, the result below states the optimality of the exponent $\frac{N+2}{N-2}$ when it comes to the study of radial solutions to \eqref{GM}.

\begin{corollary}\label{cor3}
Assume $\rho\equiv 0$. 
\begin{enumerate}
\item[{\rm (i)}] If $0<p< \frac{N+2}{N-2}$ then the system \eqref{GM} has no positive radial solutions for all $q,m,s>0$;  
\item[{\rm (ii)}] Conversely, for all $p>\frac{N+2}{N-2}$ there exist $q,m, s>0$ such that \eqref{GM} has a positive radial solution.
\end{enumerate}

\end{corollary}

The main tool in deriving the nonexistence of positive solutions to \eqref{GM0} and \eqref{GM} is the use of various integral representations of solutions to semilinear elliptic equations in $\R^N$ which were recently obtained in \cite{DG22}; we recall the relevant findings of \cite{DG22} in Section 2.1 below. The existence of positive solutions to \eqref{GM0} and \eqref{GM} combines the Schauder fixed point theorem with  the study of the single equation 
$$
\begin{cases}
-\Delta v+\mu v=\psi(x) v^{-s}\quad\mbox{ in }\; \R^N,\\
v(x)\to 0\mbox{  as }|x|\to \infty,
\end{cases}
$$ 
where $\psi\in C(\R^N)$ has a specific decay rate at infinity and $\mu\geq 0$. 

The remaining part of the manuscript is organised as follows. In Section 2 we derive some preliminary results related to integral representations for semilinear elliptic equations in $\R^N$. Sections 3 is devoted to the proof of Theorem \ref{th1}. In Section 4 we prove Theorem \ref{thgs} and Corollary \ref{cor1}. Finally, the proof of Theorem \ref{thex} and Corollaries \ref{corex}, \ref{cor2}  and \ref{cor3} are given in Section 5. 
\section{Preliminary results}

\subsection{Fundamental solutions and integral representations}

Let $s \in \R$ and $z\in \C\setminus\{w\in \C: {\rm Re\,}w\leq 0, \, {\rm Im\,}w =0\}$. The modified Bessel functions $I_s(z)$ and $K_s(z)$ of the first and second kind respectively are defined as 
\begin{equation}\label{bsl}
I_s(z)=\sum_{k=0}^\infty \frac{\Big(\displaystyle \frac{z}{2}\Big)^{2k+s}}{k!\Gamma(k+s+1)}\quad\mbox{ and }\quad
K_s(z)=\frac{\pi}{2}\cdot \frac{I_{-s}(z)-I_s(z)}{\sin(\pi s)},
\end{equation}
where $\Gamma(z)$ is the standard Gamma function. 
For an integer $n$, the above definition of $K_n(z)$ is understood in the sense $K_n(z)=\lim_{s\to n}K_{s}(z)$. 
Given $\lambda>0$,  the function
\begin{equation}\label{fbesselp}
\mathcal{G}_{\lambda}(x)=\frac{\lambda^{N-2}}{\sqrt{(2\pi)^N}} \cdot  \frac{ K_{\frac{N}{2}-1}(\lambda^{1/2}|x|)}{|x|^{\frac{N}{2}-1}}\, , \; x\in \R^N\setminus\{0\}\,,
\end{equation}
is the  fundamental solution of the operator $-\Delta+\lambda I$ (see \cite[Section 3]{B1996}) in the sense that
\begin{equation}\label{fsl}
(-\Delta+\lambda I) \mathcal{G}_{\lambda}=\delta_0\quad\mbox{ in }\mathcal{D}'(\R^N).
\end{equation}
The asymptotic behaviour of $\mathcal{G}_\lambda$ is summarised below.

\begin{lemma}\label{lem1}{\rm (see \cite[Chapter 4, pag. 416]{AS1961}, \cite[Section 2]{DG22}) } 

Let $\lambda>0$. The Bessel potential $\mathcal{G}_{\lambda}$ given by \eqref{fbesselp} has the following properties:
\begin{enumerate}
\item[\rm (i)]  There exists $c_1>1$ such that
$$
\frac{1}{c_1}  |x|^{-\frac{N-1}{2}}e^{-\sqrt{\lambda} |x|} \leq \mathcal{G}_\lambda (x)\leq c_1
|x|^{-\frac{N-1}{2}}e^{-\sqrt{\lambda} |x|} \quad\mbox{ for all }|x|>1.
$$

\item[\rm (ii)]  There exists $c_2>1$ such that
$$
\displaystyle \frac{1}{c_2}  |x|^{2-N}\leq \mathcal{G}_{\lambda}(x)\leq c_2 |x|^{2-N} 
\quad\mbox{ for all }0<|x|<1.
$$

\item[\rm (iii)] $\mathcal{G}_{\lambda}(x)$ is a decreasing function of $|x|$ and $\mathcal{G}_{\lambda}\in L^1(\R^N)$.

\item[\rm (iv)] There exists $C>0$ such that  
\begin{equation}\label{esgamma}
\Big| \nabla \mathcal{G}_\lambda (x)  \Big|\leq C
|x|^{-\frac{N-1}{2}}e^{-\sqrt{\lambda} |x|} 
\quad \mbox{  for all }|x|>1.
\end{equation}
\end{enumerate}

\end{lemma}

The following result, recently obtained in \cite{DG22},  establishes the analogy between the distributional solutions of 
\begin{equation}\label{distr}
-\Delta u+\lambda u=f(x)\quad\mbox{  in }\; \mathcal{D}'(\R^N),\quad N\geq 3,
\end{equation}
and their integral representation with kernel $\mathcal{G}_\lambda$. 

\begin{proposition}\label{prep} {\rm (see \cite[Theorem 1.7]{DG22})} 

Let $\lambda>0$ and $f\in L^1_{loc}(\R^N)$. 
If $u\in L^1_{loc}(\R^N)$ is a distributional solution of \eqref{distr} which  satisfies 
\begin{equation}\label{subexp}
\liminf_{R\to \infty}  R^{-\frac{N+1}{2} }   \int_{R\leq |x-y|\leq 2R}|u(y)| e^{-\sqrt{\lambda} |x-y|} dy=0\quad\mbox{ for a.a. }x\in \R^N,
\end{equation}
then 
\begin{equation}\label{repm0}
u(x)=\int_{\R^N} \mathcal{G}_\lambda(x-y) f(y) dy \quad\mbox{ for a.a. }x\in \R^N,
\end{equation}
where $\mathcal{G}_\lambda$ is the fundamental solution of $-\Delta+\lambda I$ as given in \eqref{fbesselp}.

Conversely, if 
\begin{equation}\label{finiteint2}
  \int_{\R^N} \mathcal{G}_\lambda(x-y) f(y)dy < \infty\quad \mbox{ for a.a. } x\in \R^N,
\end{equation}
then the function $u$ given by \eqref{repm0} is a distributional solution of \eqref{distr}.
\end{proposition}

Observe that if 
$$
\lim_{R\to \infty}  R^{-\frac{N+1}{2} }  e^{-\sqrt{\lambda} R}  \int_{R\leq |x-y|\leq 2R}|u(y)| dy=0\quad\mbox{ for a.a. }x\in \R^N,
$$
then $u$ satisfies \eqref{subexp}. This follows from the fact that $ e^{-\sqrt{\lambda} |x-y|}\leq  e^{-\sqrt{\lambda} R}$ in the annular region $\{R\leq |x-y|\leq 2R\}$. In particular, any bounded solution $u$ of \eqref{distr} admits the representation \eqref{repm0}. 

As a consequence, we have:

\begin{lemma}\label{Glem}
Let $\lambda, \mu>0$ and $(u,v)$ be a solution of \eqref{GM0}. Then, for all $x\in \R^N$ one has
\begin{subequations}
\begin{align}
u(x)&=\int_{\R^N}\mathcal{G}_\lambda(x-y)\Big[\frac{u^p(y)}{v^q(y)}+\rho(y)  \Big] dy\, , \label{eqG1} \\[0.07in]
v(x) & = \int_{\R^N}\mathcal{G}_\lambda(x-y) \frac{u^m(y)}{v^s(y)} dy. \label{eqG2}
\end{align}
\end{subequations}
\end{lemma}

To end this section, let us mention that similar results hold for the Laplace operator $-\Delta$. Precisely, let 
\begin{equation}\label{gf1}
\mathcal{G}_0=\frac{1}{4\pi^{N/2}}\Gamma\Big(\frac{N-2}{2}\Big) |x|^{2-N}.
\end{equation}
Then, $\mathcal{G}_0$ is the  fundamental solution of $-\Delta$ in the sense that
$$
-\Delta \mathcal{G}_{0}=\delta_0\quad\mbox{ in }\mathcal{D}'(\R^N).
$$
The result below was obtained in \cite{CDM08} and extended in \cite{DG22}. 

\begin{proposition}\label{prep0} {\rm (see \cite[Theorem 2.4]{CDM08}, \cite[Theorem 1.1]{DG22})} 

Let $f\in L^1_{loc}(\R^N)$, $N\geq 3$, $u:\R^N\to \R$ be a measurable function and $\ell\in \R$.  

The following statements are equivalent:

\begin{enumerate}
\item[\rm (i)] $u\in L^1_{loc}(\R^N)$ is a distributional solution of  
\begin{equation}\label{eqrep1}
-\Delta u=f(x)\quad\mbox{ in  }\; \mathcal{D}'(\R^N)
\end{equation}
and for a.a. $x\in \R^N$, $u$ satisfies 
\begin{equation}\label{rep0}
\liminf_{R\to \infty} \frac{1}{R^N}\int\limits_{R\leq |y-x|\leq 2R}|u(y)-\ell|dy=0.
\end{equation}

\item[\rm (ii)] $u\in L^1_{loc}(\R^N)$ is a distributional solution of \eqref{eqrep1}, ${\rm essinf}_{\R^N}\,u=\ell$ and 
$$
-\Delta u\ge 0\quad\mbox{ in }\quad \mathcal{D}'(\R^N).
$$

\item[\rm (iii)] We have
\begin{equation}\label{finiteint}
  \int_{\R^N}\mathcal{G}_0(x-y) f(y) dy < \infty\quad \mbox{ for a.a. } x\in \R^N
\end{equation}
and the following representation holds
\begin{equation}\label{rep2}
u(x)=\ell+ \int_{\R^N} \mathcal{G}_0 (x-y) f(y) dy \quad\mbox{ for a.a. }x\in \R^N.
\end{equation}
\end{enumerate}

\end{proposition}

As a consequence of Proposition \ref{prep0}  and similar to Lemma \ref{Glem} we have:

\begin{lemma}\label{Glem0}
Let $(u,v)$ be a solution of \eqref{GM}. Then, there exists a constant $c(N)>0$ such that for all $x\in \R^N$ one has
\begin{subequations}
\begin{align}
u(x)&=c(N)\int_{\R^N}|x-y|^{2-N}\Big[\frac{u^p(y)}{v^q(y)}+\rho(y)  \Big] dy\, , \label{eqG01} \\[0.07in]
v(x) & = c(N)\int_{\R^N} |x-y|^{2-N} \frac{u^m(y)}{v^s(y)} dy. \label{eqG02}
\end{align}
\end{subequations}
\end{lemma}

\subsection{Some results on semilinear elliptic equations}

In this section we collect some  useful results on semilinear elliptic equations. 

\begin{lemma}{\rm (see \cite[Theorem 2.1]{AS11})}\label{arms}
Let $f:(0, \infty)\to (0, \infty)$ be such that 
$$
\liminf_{t\to 0}  f(t)t^{-\frac{N}{N-2}}>0.
$$ 
Then, the inequality $-\Delta u\geq f(u)$ has no positive $C^2$-solutions in any exterior domain of $\R^N$, $N\geq 3$. 
\end{lemma}

\begin{lemma}{\rm (see \cite[Theorem 4.1]{GS81})}\label{gspruck}
Let $u\in C^2(\R^N) $, $N\geq 3$,  be a nonnegative solution of 
$$
-\Delta u=h(x) u^{p}\quad\mbox{ in }\R^N. 
$$
Assume $1\leq p<\frac{N+2}{N-2}$ and $h\in C^2(\R^N)$ satisfies
\begin{enumerate}
\item[\rm (i)] $\Delta h\geq 0$ in $\R^N$.
\item[\rm (ii)] $h(x)\geq c$ for $|x|>1$ large, for some constant $c>0$.
\item[\rm (iii)] There exists $C>0$ such that 
$\displaystyle |\nabla h(x)|\leq \frac{C}{|x|} h(x)$ for $|x|>1$ large.
\end{enumerate}
Then $u\equiv 0$. 
\end{lemma}

\begin{lemma}\label{funcw}
For $a>0$ denote $W_a(x)=e^{-a\sqrt{1+|x|^2}}$. Then  
\begin{equation}\label{WA}
 (\lambda -a^2) W_a \leq -\Delta W_a+\lambda W_a(x)\leq (\lambda +Na)W_a\quad\mbox{ in }\R^N.
\end{equation}
\end{lemma}
\begin{proof}
By direct calculation we have 
$$
-\Delta W_a(x)+\lambda  W_a(x)=\Big( \lambda-a^2+\frac{a^2}{1+|x|^2}+\frac{a}{\sqrt{(1+|x|^2})^3}+  \frac{(N-1)a}{\sqrt{1+|x|^2}}\Big)W_a(x).
$$
Now, the estimate \eqref{WA} follows immediately.
\end{proof}

\begin{lemma}\label{eax}
Let $s>0$ and $\psi\in C(\R^N)$ be a positive function such that $\psi(x)\simeq e^{-\gamma |x|}$ for some $\gamma>0$. Then, for any $\mu>\big(\frac{\gamma}{s+1}\big)^2$ the problem
\begin{equation}\label{singp}
\begin{cases}
\displaystyle   -\Delta v+\mu v=\psi(x) v^{-s} \quad\mbox{ in }\R^N,\\
v(x)\to 0\mbox{ as }|x|\to \infty,
\end{cases}
\end{equation}
has a unique solution $v\in C^2(\R^N)$. Moreover, $v(x)\simeq W_{\frac{\gamma}{s+1}} (x)$.
\end{lemma}

\begin{proof} Since $\psi(x)\simeq e^{-\gamma |x|}$, we can find $M>m>0$ such that 
\begin{equation}\label{mgg}
mW_\gamma(x)\leq \psi(x)\leq MW_\gamma(x)\quad\mbox{ for all }x\in \R^N.
\end{equation}
Define now 
$$
a=\frac{\gamma}{s+1}\, ,\quad C=\Big(\frac{M}{\mu-a^2}  \Big)^{\frac{1}{s+1}} \, ,\quad c=\Big(\frac{m}{\mu+Na}  \Big)^{\frac{1}{s+1}}.
$$ 
Then, $\mu>\big(\frac{\gamma}{s+1}\big)^2$ implies  $\mu>a^2$ and $C>c>0$. Define now 
$\underline v=cW_a$ and $\overline v=CW_a$. Using the estimates \eqref{WA} together with \eqref{mgg} we find 
$$
-\Delta \underline v+\mu \underline v\leq c(\mu+Na)W_a\leq mW_\gamma \, \underline v^{-s}\leq \psi(x) \underline v^{-s} \quad\mbox{ in }\R^N,
$$
$$
-\Delta \overline v+\mu \overline v\geq C(\mu-a^2)W_a\geq  M W_\gamma \, \overline v^{-s}\geq \psi(x) \overline v^{-s} \quad\mbox{ in }\R^N.
$$
Thus $(\underline v, \overline v)$ is an ordered pair of sub and super-solution of \eqref{singp}.  
By the standard sub and super-solution method (see, e.g., \cite{GR08book}) for any $n\geq 1$ there exists $v_n\in C^2(\overline B_n)$ such that 
\begin{equation}\label{ss1}
\begin{cases}
-\Delta  v_n+\mu v_n = \psi(x) v_n^{-s} &\quad\mbox{ in }B_n,\\[0.05in]
\; \; v_n=\underline v &\quad\mbox{ on }\partial B_n,\\[0.05in]
\;\; \underline v\leq v_n\leq \overline v&\quad\mbox{ in }B_n.
\end{cases}
\end{equation}
We claim that $v_n\leq v_{n+1}$ in $\overline B_n$. The inequality clearly holds on $\partial B_n$. Assuming that the set $\{x\in B_n: v_n(x)>v_{n+1}(x)\}$ is nonempty, then the maximum of $v_n-v_{n+1}$ on $\overline B_n$ occurs in $B_n$. At this maximum point, say $x_0\in B_n$, we have $(v_n-v_{n+1})(x_0)>0$ and $\Delta (v_n-v_{n+1})(x_0)\leq 0$.  
$$
0\leq -\Delta (v_n-v_{n+1})(x_0)+\mu  (v_n-v_{n+1})(x_0)\leq \psi(x_0)\Big(v_n(x_0)^{-s}-v_{n+1}(x_0)^{-s}\Big)<0,
$$
contradiction. Thus, $\{v_n(x)\}_{n\geq 1}$ is an increasing sequence bounded from above and from below by $\overline v(x)$ and $\underline v(x)$. We can define $v(x)=\lim_{n\to \infty}v_n(x)$. By standard elliptic arguments we further deduce that $v$ is a solution of \eqref{singp} while  $\underline v\leq v_n\leq \overline v$ yields $v(x)\simeq W_{\frac{\gamma}{s+1}}(x)$. The uniqueness of $v$ as a solution of \eqref{singp} follows in the same way as above. 
\end{proof}

\begin{lemma}\label{funcz}
For $a>0$ denote $Z_a(x)=(1+|x|^2)^{-a/2}$. Then  
\begin{equation}\label{ZA}
 a(N-a-2) Z_{a+2} \leq -\Delta Z_a(x) \leq aN Z_{a+2}\quad\mbox{ in }\R^N.
\end{equation}
\end{lemma}
\begin{proof}
By direct calculation we have 
$$
-\Delta Z_a(x)=a\big\{N+(N-a-2)|x|^2\big\} Z_{a+4}(x) \quad\mbox{ in }\R^N
$$
and
$$
(N-a-2)(|x|^2+1)\leq N+(N-a-2)|x|^2\leq N (|x|^2+1) \quad\mbox{ in }\R^N.
$$
Combining the above estimates we deduce \eqref{ZA}.
\end{proof}

We next establish a counterpart result to Lemma \ref{eax} for the problem
\begin{equation}\label{singpz}
\begin{cases}
\displaystyle   -\Delta v=\phi(x) v^{-s} \quad\mbox{ in }\R^N,\\
v(x)\to 0\mbox{ as }|x|\to \infty.
\end{cases}
\end{equation}

\begin{lemma}\label{zz}
Let $s>0$  and $\phi\in C(\R^N)$ be a positive function.  
\begin{enumerate}
\item[{\rm (i)}] Assume 
\begin{equation}\label{eqga}
\phi(x)\geq c|x|^{-\gamma}\quad\mbox{ in }\; \R^N\setminus B_1,
\end{equation}
for $c>0$ and $0<\gamma\leq 2$. Then \eqref{singpz} has no positive $C^2$-solutions.
\item[{\rm (ii)}] Assume $\phi\simeq |x|^{-\gamma}$ where $2<\gamma<(N-2)s+N$. Then \eqref{singpz} 
has a unique positive solution $v\in C^2(\R^N)$. Moreover, $v(x)\simeq Z_{\frac{\gamma-2}{s+1}}(x)$.
\end{enumerate}
\end{lemma}

\begin{proof} (i) Since $v(y)\to 0$ as $|y|\to \infty$,  condition \eqref{rep0} with $\ell=0$ holds. Also, $v^{-s}(y)\geq c$ in $\R^N\setminus B_1$, where $c>0$ is a constant.  

Let $x\in \R^N$, $|x|>1$. By Proposition \ref{prep0}  we have 
$$
v(x)=\int_{\R^N}\mathcal{G}_0(x-y) \phi(y) v^{-s}(y) dy\geq C\int_{|y|>|x|}|x-y|^{2-N} |y|^{-\gamma}  ds.
$$
Note that for $|y|>|x|$ we have $|x-y|\leq 2|y|$ so the above estimate yields
$$
\begin{aligned}
v(x) &\geq C\int_{|y|>|x|} |y|^{2-N-\gamma} dy=C\int_{|x|}^\infty t^{1-\gamma} dt=\infty,
\end{aligned}
$$
since $0< \gamma\leq 2$.  Hence, \eqref{singpz} has no positive solutions.

\medskip

(ii) From our assumptions, there exist $M>m>0$ such that
$$
m Z_{\gamma}(x)\leq \phi(x)\leq M Z_\gamma(x)\quad\mbox{for all }x\in \R^N.
$$
Take now 
$$
a=\frac{\gamma-2}{s+1}\,, \quad C=\Big( \frac{M}{a(N-a-2)}\Big)^{\frac{1}{s+1}}\,,\quad c=\Big(\frac{m}{ aN}\Big)^{\frac{1}{s+1}}.
$$
Since $2<\gamma<(N-2)s+N$, we have $a\in (0, N-2)$ and $C,c>0$. By taking $M>1$ large and $m\in (0,1)$ small, we may assume $C>c>0$. 
By Lemma \ref{funcz} one can check that $\underline v(x)=cZ_a$ and $\overline v(x)=CZ_a$ satisfy 
$$
\begin{cases}
-\Delta \underline v\leq \phi(x) \underline v^{-s} \quad\mbox{ in }\R^N,\\[0.05in]
-\Delta \overline v\geq \phi(x) \overline v^{-s} \quad\mbox{ in }\R^N.
\end{cases}
$$
We next proceed as in Lemma \ref{eax}. First, for any $n\geq 1$, by the sub and super-solution method  there exists $v_n\in C^2(\overline B_n)$ such that 
$$
\begin{cases}
-\Delta  v_n = \phi(x) v_n^{-s} &\quad\mbox{ in }B_n,\\[0.05in]
\; \; v_n=\underline v &\quad\mbox{ on }\partial B_n,\\[0.05in]
\;\; \underline v\leq v_n\leq \overline v&\quad\mbox{ in }B_n.
\end{cases}
$$
Then, $v_n\leq v_{n+1}$ in $\overline B_n$. Finally, using standard elliptic arguments one has that $v(x)=\lim_{n\to \infty}v_n(x)$ is a solution of  \eqref{singpz}. The uniqueness follows as in the proof of Lemma \ref{eax}.
\end{proof}

\section{Proof of Theorem \ref{th1}}\label{pt1}

(i) Assume $0<p\leq 1$ and there exists $(u,v)$ a positive solution of the system \eqref{GM0}. Since $u(x)\to 0$ as $|x|\to \infty$, we have for $R>0$ large enough that $u^p\geq u$ in $\R^N\setminus B_R$. Let $\lambda_1(R)$ denote the first eigenvalue of $-\Delta$ in $\Omega_R:=B_{2R}\setminus \overline B_R$ subject to homogeneous Dirichlet boundary condition on $\partial \Omega_R$. Denote also by $\varphi_1\in C^2(\overline \Omega_R)$ the corresponding eigenfunction which we normalize as $\varphi_1>0$ in $\Omega_R$, that is
$$
\begin{cases}
-\Delta \varphi_1=\lambda_1(R) \varphi_1\,, \;\; \varphi_1>0 &\quad\mbox{ in }\Omega_R,\\
\;\;\; \varphi_1=0&\quad\mbox{ on }\partial\Omega_R.
\end{cases}
$$ 
Then, multiplying by $\varphi_1$ in the first equation of \eqref{GM0} and integrating over $\Omega_R$ we find
\begin{equation}\label{ph1}
-\int_{\Omega_R}\varphi_1 \Delta u\geq \int_{\Omega_R} (v^{-q}-\lambda)u\varphi_1 \geq \inf_{\overline \Omega_R} (v^{-q}-\lambda) \int_{\Omega_R} u\varphi_1.
\end{equation}
Using the Green's formula and the Hopf boundary point lemma we have
\begin{equation}\label{ph2}
-\int_{\Omega_R}\varphi_1 \Delta u=-\int_{\Omega_R} u \Delta \varphi_1 +\int_{\partial\Omega_R} \Big(u \frac{\partial \varphi_1}{\partial\nu}- \varphi_1 \frac{\partial u}{\partial\nu}\Big)\leq \lambda_1(R)\int_{\Omega_R} u\varphi_1,
\end{equation}
where $\nu$ denotes the outward unit normal vector to $\partial\Omega_R$. From \eqref{ph1} and \eqref{ph2} we find 
$$
\lambda_1(R)  \geq \inf_{\overline \Omega_R} (v^{-q}-\lambda).
$$
This yields a contradiction for $R>0$ large since $\lambda_1(R)\simeq R^{-2}$ as $R\to \infty$, while $\inf_{\overline \Omega_R} (v^{-q}-\lambda)\to \infty$ as $R\to \infty$. 

\medskip

\noindent (ii) Assume $p>1$, $\sigma>1$, $\mu> \big(\frac{m}{s+1}\big)^2\lambda$ and there exists $(u,v)$ a positive solution of \eqref{GM0} with $u(x)\simeq e^{-a|x|}$ for some $a>0$. Note that the fundamental solution $\mathcal{G}_\lambda$ of the operator $-\Delta +\lambda I$ satisfies \eqref{fsl}. By the maximum principle one has $u\geq c\mathcal{G}_\lambda$ in $\R^N\setminus B_1$, for some $c>0$. Using Lemma \ref{lem1}(i) it follows that 
$$
u(x)\geq c|x|^{-\frac{N-1}{2}}e^{-\sqrt{\lambda}|x|}\quad\mbox{ for all }|x|>1.
$$
Since $u(x)\simeq e^{-a|x|}$, it follows that $0<a\leq \sqrt{\lambda}$. 

Let us note that $v$ satisfies \eqref{singp} with $\psi(x)=u^m(x)\simeq e^{-\gamma|x|}$ where $\gamma=am$. We are now entitled to apply Lemma \ref{eax} since $a\leq \sqrt{\lambda}$ and by our hypothesis
$$
\mu>\Big(\frac{m}{s+1}\Big)^2\lambda\geq \Big(\frac{m a}{s+1}\Big)^2=\Big(\frac{\gamma}{s+1}\Big)^2.
$$ 
By Lemma \ref{eax} we deduce $v(x)\simeq e^{-\frac{am}{s+1}|x|}$. This yields
$$
c_1e^{\beta |x|}\leq  \frac{u^p(x)}{v^q(x)}\leq c_2e^{\beta |x|}\quad\mbox{ for all }x\in \R^N,
$$
where $c_2>c_1>0$ and $\beta=a\big(\frac{mq}{s+1}-p\big)$.
Using this fact in \eqref{eqG1}, by the estimates in Lemma \ref{lem1}(i) we find 
\begin{equation}\label{esta}
\begin{aligned}
u(x) &\geq \int_{\R^N}\mathcal{G}_\lambda(x-y) \frac{u^p(y)}{v^q(y)} dy\\[0.01in]
&\geq c \int_{\R^N}\mathcal{G}_\lambda(x-y) e^{\beta |y|} dy\\[0.01in]
&\geq c \int_{|x-y|>1} |x-y|^{-\frac{N-1}{2}}e^{-\sqrt{\lambda} |x-y|} \, e^{\beta |y|} dy.
\end{aligned}
\end{equation}
If $\beta\geq 0$ then $\beta|y|\geq \beta|x|-\beta|x-y|$ and \eqref{esta} yields
$$
\begin{aligned}
u(x) &  \geq c e^{\beta|x|}   \int_{|x-y|>1} |x-y|^{-\frac{N-1}{2}}e^{-(\sqrt{\lambda}+\beta) |x-y|} \, dy\\[0.01in]
&= c e^{\beta|x|}   \int_{|z|>1} |z|^{-\frac{N-1}{2}}e^{-(\sqrt{\lambda}+\beta) |z|} \, dy \\[0.01in]
&\geq c  e^{\beta|x|},
\end{aligned}
$$
which contradicts the fact that $u(x)\to 0$ as $|x|\to \infty$. 

\noindent If $\beta< 0$ then $\beta |y|\geq \beta|x|+\beta|x-y|$ and similar to above, from \eqref{esta} we deduce 
$$
\begin{aligned}
u(x) &  \geq c e^{\beta|x|}   \int_{|x-y|>1} |x-y|^{-\frac{N-1}{2}}e^{-(\sqrt{\lambda}-\beta) |x-y|} \, dy\\[0.01in]
&= c e^{\beta|x|}   \int_{|z|>1} |z|^{-\frac{N-1}{2}}e^{-(\sqrt{\lambda}-\beta) |z|} \, red dy \geq c  e^{\beta|x|} .
\end{aligned}
$$
Since $u(x)\simeq e^{-a|x|}$, it follows that $-a\geq \beta$, that is,
$$
p-1\geq \frac{mq}{s+1}\Longleftrightarrow \sigma\leq 1, \mbox{ contradiction.}
$$ 

\medskip

\noindent (iii) By our hypothesis, there exist $0<\alpha<\beta$ such that  
\begin{equation}\label{roro1}
\alpha W_a\leq \rho(x)\leq \beta W_a\quad\mbox{ in }\;  \R^N,
\end{equation}
where, as in Lemma \ref{funcw}, we denote $W_\gamma(x)=e^{-\gamma\sqrt{1+|x|^2}}$ for $\gamma>0$.
 
Let $b=am/(s+1)$. Then, since $\sigma\leq 1$, one has
\begin{equation}\label{siggg}
ap-bq\geq a.
\end{equation}
Take now $\lambda, \mu>0$ large such that
\begin{equation}\label{eqlm}
\lambda>\max\{2a^2, N^2\}\;, \quad \mu>\max\left\{2b^2, \Big(\frac{am}{s+1}\Big)^2, N^2\right\}.
\end{equation} 
For $\overline M_1>\underline M_1>0$ and $\overline M_2>\underline M_2>0$ define the set 
\begin{equation}\label{eqss}
\Sigma=\left\{u,v\in C(\R^N): 
\begin{aligned}
\underline M_1 W_a& \leq u\leq  \overline M_1 W_a\\
\underline M_2 W_b & \leq v\leq  \overline M_2 W_b
\end{aligned} \quad\mbox{ in }\R^N \right\}.
\end{equation}
For $(u,v)\in \Sigma$, let $Tu, Tv\in C^2(\R^N)$ be the unique solutions of 
\begin{equation}\label{eqw}
\begin{cases}
\displaystyle   -\Delta Tu+\lambda Tu=\frac{u^p}{v^q}+\rho(x) \quad\mbox{ in }\R^N\, , \\[0.1in]
\displaystyle   -\Delta Tv+\mu Tv=\frac{u^m}{(Tv)^s} \quad\mbox{ in }\R^N,\\[0.1in]
Tu, Tv\to 0\mbox{ as }|x|\to \infty.
\end{cases}
\end{equation}
Let us first note that for $(u,v)\in \Sigma$ and \eqref{siggg} we have 
\begin{equation}\label{wa1}
\begin{aligned}
\frac{u^p}{v^q}+\rho(x) & \leq \overline M_1^p \underline M_2^{-q} W_a^p W_b^{-q}+\beta W_a\\
& =\overline M_1^p \underline M_2^{-q} W_{ap-bq} +\beta W_a\\
& \leq \Big( \overline M_1^p \underline M_2^{-q} +\beta\Big) W_a\quad\mbox{ in }\R^N.
\end{aligned}
\end{equation}
This shows that $f=\frac{u^p}{v^q}+\rho(x)$ is bounded and thus \eqref{finiteint2} holds. By Proposition \ref{prep} we have 
$$
Tu(x)=\int_{\R^N} \mathcal{G}_\lambda(x-y)\Big[\frac{u^p(y)}{v^q(y)}+\rho(y)  \Big] dy\quad\mbox{ for all }x\in \R^N,
$$
and this justifies the existence and uniqueness of the solution $Tu$ in \eqref{eqw}.  
Also, from Lemma \ref{eax} with $\psi(x)=u^m(x)\simeq e^{-\gamma |x|}$, $\gamma=am$ and $\mu>(\frac{am}{s+1})^2$ (see \eqref{eqlm}) we deduce the existence and the uniqueness of $Tv\in C^2(\R^N)$. 

Define $\Psi:\Sigma\to C(\R^N)\times C(\R^N)$ by $\Psi(u,v)=(Tu, Tv)$ and observe that all fixed points of $\Psi$ are solutions of \eqref{GM0}. We shall establish the existence of such fixed points (and thus, the existence of a solution to \eqref{GM0}) through Schauder fixed point theorem.

\begin{lemma}\label{MM}
If $\overline M_1>\underline M_1>0$ and $\overline M_2>\underline M_2>0$ satisfy 
\begin{subequations}
\begin{align}
\frac{\lambda}{2} \overline M_1&\geq \overline M_1^p \underline M_2^{-q}+\beta \, , \label{eqm1} \\[0.07in]
2\lambda \underline M_1 & =\alpha \, , \label{eqm2}\\
\frac{\mu}{2} \overline M_2&=\overline M_1^m \overline M_2^{-s} \, , \label{eqm3}\\ 
2\mu \underline M_2 & = \underline M_1^m \underline M_2^{-s}\, , \label{eqm4}
\end{align}
\end{subequations}
then $\Psi(\Sigma)\subset \Sigma$. 
\end{lemma}
\begin{proof} From the estimate \eqref{wa1} we have
$$
-\Delta Tu+\lambda Tu\leq \Big( \overline M_1^p \underline M_2^{-q} +\beta\Big) W_a \quad\mbox{ in }\R^N.
$$
Thus, by Lemma \ref{funcw} and \eqref{eqlm}, the function 
$\displaystyle 
\overline \zeta=\frac{2}{\lambda} \Big( \overline M_1^p \underline M_2^{-q} +\beta\Big) W_a
$
satisfies
$$
\begin{cases}
-\Delta \overline \zeta+\lambda \overline \zeta \geq (\lambda-a^2) \overline \zeta \geq \frac{\lambda}{2}\overline \zeta=\Big( \overline M_1^p \underline M_2^{-q} +\beta\Big) W_a \geq -\Delta Tu+\lambda Tu  \quad\mbox{ in }\R^N\, ,\\[0.05in]
\;\;\; Tu(x)\, , \overline \zeta(x)\to 0\quad\mbox{ as }|x|\to \infty.
\end{cases}
$$
By the maximum principle and \eqref{eqm1} we find $Tu\leq  \overline\zeta \leq \overline M_1 W_a$ in $\R^N$. Also,
$$
-\Delta Tu+\lambda Tu\geq \rho(x) \geq \alpha W_a  \quad\mbox{ in }\R^N.
$$
Thus, by Lemma \ref{funcw} and the above estimate the function $\displaystyle \underline \zeta=\frac{\alpha}{2\lambda} W_a$ satisfies 
$$
\begin{cases}
\displaystyle -\Delta  \underline \zeta +\lambda  \underline \zeta \leq (\lambda+Na)\underline \zeta\leq  2\lambda \underline \zeta =  \alpha W_a\leq -\Delta Tu+\lambda Tu \quad\mbox{ in }\R^N\, ,\\[0.05in]
\;\;\; Tu(x)\, ,  \underline \zeta (x)\to 0\quad\mbox{ as }|x|\to \infty.
\end{cases}
$$
By the maximum principle and \eqref{eqm2} one has $Tu\geq \underline \zeta=\underline M_1 W_a$ in $\R^N$. 

We now focus on the inequality $\underline M_2 W_b\leq Tv\leq \overline M_2 W_b$ in $\R^N$ which we will achieve using \eqref{eqm3}-\eqref{eqm4}. 
From  $\underline M_1 W_a\leq u \leq \overline M_1 W_a$ in $\R^N$, one has that $Tv$ satisfies 
\begin{equation}\label{AA}
\underline M_1^m W_a^m (Tv)^{-s} \leq -\Delta Tv+\mu Tv\leq \overline M_1^m W_a^m (Tv)^{-s}\quad\mbox{ in }\R^N.
\end{equation}
We first note that $\overline \xi=\overline M_2 W_b$ satisfies by \eqref{eqm3}, \eqref{eqlm} and the estimate \eqref{WA},  the inequality
$$
-\Delta \overline \xi +\mu \overline \xi \geq (\mu-b^2)\overline \xi\geq \frac{\mu}{2} \overline M_2 W_b= \overline M_1^m W_a^m \overline \xi^{-s}\quad \mbox{ in }\R^N.
$$
We have thus obtained
$$
\begin{cases}
-\Delta Tv+\mu Tv\leq \overline M_1^m W_a^m (Tv)^{-s}   \quad\mbox{ in }\R^N\, ,\\[0.07in]
-\Delta \overline \xi +\mu \overline \xi \geq  \overline M_1^m W_a^m \overline \xi^{-s}\quad \mbox{ in }\R^N\, , \\[0.05in]
\;\;\; T v(x)\, ,  \overline \xi (x)\to 0\quad\mbox{ as }|x|\to \infty.
\end{cases}
$$
With the same method as in the proof of the uniqueness of the solution in Lemma \ref{eax} we deduce $Tv \leq \overline \xi=\overline M_2 W_b$ in $\R^N$.  Finally, from \eqref{eqm4}, \eqref{eqlm} and Lemma \ref{funcw} we have that $\underline \xi=\underline M_2 W_b$ satisfies 
\begin{equation}\label{BB}
-\Delta \underline \xi +\mu \underline \xi \leq (\mu+Nb)\underline \xi\leq 2\mu \underline M_2 W_b= \underline M_1^m W_a^m \underline \xi^{-s}\quad \mbox{ in }\R^N.
\end{equation}
As above, from \eqref{AA} and \eqref{BB} we deduce $Tv\geq \underline \xi=\underline M_2 W_b$ in $\R^N$ and this concludes our proof.
\end{proof}

\begin{lemma}\label{const}
If $\lambda, \mu>0$ satisfy \eqref{lamm} for some $C_1, C_2>0$, then there exist  $\overline M_1>\underline M_1>0$ and $\overline M_2>\underline M_2>0$ that satisfy \eqref{eqm1}-\eqref{eqm4} and thus $\Sigma$ is an invariant set of $\Psi$.
\end{lemma}
\begin{proof}
We shall prove that if \eqref{lamm} holds for some suitable constants $C_1, C_2>0$, then there are $\overline M_1>\underline M_1>0$ and $\overline M_2>\underline M_2>0$ that satisfy slightly more general conditions than \eqref{eqm1}-\eqref{eqm4}, namely
\begin{subequations}
\begin{align}
\frac{\lambda}{4} \overline M_1&= \overline M_1^p \underline M_2^{-q} \quad\mbox{ and }\quad \frac{\lambda}{4} \overline M_1\geq \beta \, , \label{eqm1a} \\[0.07in]
2\lambda \underline M_1 & =\alpha \, , \label{eqm2a}\\
\frac{\mu}{2} \overline M_2&=\overline M_1^m \overline M_2^{-s} \, , \label{eqm3a}\\ 
2\mu \underline M_2 & = \underline M_1^m \underline M_2^{-s}. \label{eqm4a}
\end{align}
\end{subequations}
Adding the two conditions in \eqref{eqm1a} we deduce \eqref{eqm1}. Let us note that from \eqref{eqm2a} we have $\underline M_1=\frac{\alpha}{2}\lambda^{-1}$. Using this in \eqref{eqm4a} we deduce 
\begin{equation}\label{eqm2b}
\underline M_2 =\left(\frac{\alpha^m}{2^{m+1}}  \right)^{\frac{1}{s+1}} \mu^{-\frac{1}{s+1}}\lambda^{-\frac{m}{s+1}}. 
\end{equation}
Further, from \eqref{eqm1a}$_1$ we find
\begin{equation}\label{eqm2c}
\overline M_1 =\left[ \frac{1}{4}\left(\frac{\alpha^m}{2^{m+1}}  \right)^{\frac{q}{s+1}} \right]^{\frac{1}{p-1}} \mu^{-\frac{\sigma}{m}} \, \lambda^{\frac{1}{p-1}-\sigma},
\end{equation}
and then from \eqref{eqm3a} one obtains
\begin{equation}\label{eqm2d}
\overline M_2 =\left\{ 2\left[ \frac{1}{4}\left(\frac{\alpha^m}{2^{m+1}}  \right)^{\frac{q}{s+1}} \right]^{\frac{m}{p-1}} \right\}^{\frac{1}{s+1}}  \mu^{-\frac{1}{s+1}-\frac{\sigma}{s+1}} \, \lambda^{\frac{m}{(p-1)(s+1)}-\frac{\sigma m}{s+1} }.
\end{equation}
Let us note that by \eqref{eqm2c}, condition \eqref{eqm1a}$_2$ holds if 
$$
\frac{1}{4\beta} \left[ \frac{1}{4}\left(\frac{\alpha^m}{2^{m+1}}  \right)^{\frac{q}{s+1}} \right]^{\frac{1}{p-1}} \mu^{-\frac{\sigma}{m}} \, \lambda^{\frac{p}{p-1}-\sigma}\geq 1 \; \Longrightarrow \mu\leq c_0\lambda^{\frac{p(s+1)}{q}-m},
$$
where $c_0>0$ depends on $p,q,m,s$ and $\alpha, \beta$. The same condition but with different constants instead of $c_0$ arises if we impose $\overline M_1>\underline M_1$ and then $\overline M_2>\underline M_2$. This finishes the proof of our lemma.
\end{proof}

\medskip

\noindent{\bf Proof of Theorem \ref{th1} completed.} Assume that \eqref{eqm1a}-\eqref{eqm4a} hold. By the above two lemmas, this implies that $\Sigma$ is invariant to $\Psi$. Due to lack of compactness, we cannot apply the Schauder fixed point theorem directly on $C(\R^N)\times C(\R^N)$. In turn, we shall use this result on $C(\overline B_n)\times C(\overline B_n)$ and combine it with elliptic regularity to deduce the existence of a positive solution to \eqref{GM0} in the whole space. 

Let 
$$
\Sigma_n=\left\{u,v\in C(\overline B_n): 
\begin{aligned}
\underline M_1 W_a& \leq u\leq  \overline M_1 W_a\\
\underline M_2 W_b & \leq v\leq  \overline M_2 W_b
\end{aligned} \quad\mbox{ in } \overline B_n \right\},
$$
where $\overline M_1>\underline M_1>0$ and $\overline M_2>\underline M_2>0$ are the constants that satisfy \eqref{eqm1a}-\eqref{eqm4a}.
\smallskip

For $(u,v)\in \Sigma_n$, let $(\widetilde u, \widetilde v)\in \Sigma$ be a fixed extension of $(u, v)$, where $\Sigma$ is defined in \eqref{eqss}. Let $(U, V)=(T\widetilde u, T\widetilde v)$ where $T\widetilde u$ and $T\widetilde v$ are given by \eqref{eqw} (with $u, v$ replaced by $\widetilde u$, $\widetilde v$).  We define
$$
\Psi_n:\Sigma_n\to C(\overline B_n)\times C(\overline B_n)\quad\mbox{ by }\quad  
\Psi_n(u, v):=\Big(U\!\mid_{B_n}, V\!\mid_{B_n}\Big).
$$
Note that by \eqref{eqw} we have that $(U, V)$ satisfies 
$$
\begin{cases}
\displaystyle   -\Delta U+\lambda U=\frac{u^p}{v^q}+\rho(x) \quad\mbox{ in }B_n\, , \\[0.1in]
\displaystyle   -\Delta V+\mu V=\frac{u^m}{ V^s} \quad\mbox{ in }B_n.\\[0.1in]
\end{cases}
$$
In particular, one has that the above definition of $\Psi_n$  is independent on the extension $(\widetilde u, \widetilde v)\in \Sigma$ of $(u, v)\in \Sigma_n$. Also, since $\Sigma$ is invariant to $\Psi$ it follows that $\Sigma_n$ is invariant to $\Psi_n$. 
Finally, $\Psi_n$ maps $\Sigma_n$ into $C^2(\overline B_n)\times C^2(\overline B_n)$ which is compactly embedded into $C(\overline B_n)\times C(\overline B_n)$; thus, $\Psi_n:\Sigma_n\to \Sigma_n$  has compact image.

We can thus apply  the Schauder fixed point theorem for $\Psi_n$ and there exists $(u_n, v_n)\in \Sigma_n$ a fixed point of $\Psi_n$, which means $(u_n, v_n)$ is a solution of  the system 
\begin{equation}\label{GM00}
\begin{cases}
\displaystyle   -\Delta u_n+\lambda u_n=\frac{u_n^p}{v_n^q}+\rho(x) \\[0.15in]
\displaystyle   -\Delta v_n+\mu v_n=\frac{u_n^m}{v_n^s}  \\[0.15in]
\underline M_1 W_a \leq u_n\leq  \overline M_1 W_a \\[0.05in]  
\underline M_2 W_b  \leq v_n\leq  \overline M_2 W_b 
\end{cases}
\quad\mbox{ in } B_n.
\end{equation}
Since $b=am/(s+1)$ and $\sigma\leq 1$ we have
\begin{equation}\label{regl}
\frac{u_n^p}{v_n^q}\leq C W_{ap-bq}\leq C W_a\, , \quad \frac{u_n^m}{v_n^s} \leq CW_{am-bs}=CW_b \quad \mbox{ in } B_n.
\end{equation}
By standard elliptic estimates, it follows that $\{(u_n, v_n)\}_{n\geq 1}$ is bounded in $W^{2,r}(B_1)\times W^{2,r}(B_1)$ for all $r>1$, and thus, it is bounded in $C^{1,\gamma}(B_1)\times C^{1,\gamma}(B_1)$ for some $\gamma\in (0,1)$. Since the embedding $C^{1,\gamma}(B_1) \hookrightarrow C^{1}(B_1)$ is compact, it follows that $\{(u_n, v_n)\}_{n\geq 1}$ has a convergent subsequence $\{(u^1_n, v^1_n)\}_{n\geq 1}$ in $C^{1}(B_1)\times C^1(B_1)$. Now, we proceed similarly with the sequence $\{(u^1_n, v^1_n)\}_{n\geq 2}$ which, by elliptic estimates it is bounded in $W^{2,r}(B_2)\times W^{2,r}(B_2)$ for all $r>1$ and thus, has a convergent subsequence $\{(u^2_n, v^2_n)\}_{n\geq 2}$ in $C^{1}(B_2)\times C^{1}(B_2)$. 

Inductively, we obtain a sequence $\{(u^m_n, v^m_n)\}_{n\geq m}\subset \{(u^{m-1}_n, v^{m-1}_n)\}_{n\geq m}$ which is convergent in $C^{1}(B_m)\times C^{1}(B_m)$. Now, the diagonal sequence $\{(u^n_n, v^n_n)\}_{n\geq 1}$ will be convergent in $C^{1}_{loc}(\R^N)\times C^{1}_{loc}(\R^N)$. Its limit $(u,v)$ is a $C^1$-solution of \eqref{GM0} and from \eqref{regl} it satisfies
$$
\frac{u^p}{v^q}\leq C  W_a\, , \quad \frac{u^m}{v^s} \leq CW_b \quad \mbox{ in } \R^N.
$$
By regularity theory one has $(u, v)\in C^{2}(\R^N)\times C^{2}(\R^N)$. This concludes our proof.
\qed

\section{Proof of Theorem \ref{thgs}}
Assume $(u,v)$ is a positive solution of \eqref{GM}. 

\medskip

\noindent (i) Since $v$ is positive and $v(x)\to 0$ as $|x|\to \infty$,  it follows that $v^{-q}\geq C>0$ in $\R^N$ for some constant $C>0$. Thus, $u$ satisfies $-\Delta u\geq Cu^p$ in $\R^N$. By Lemma \ref{arms}  it follows that \eqref{GM} has no solutions for all $0<p\leq N/(N-2)$. 

Assume now $0<m\leq 2/(N-2)$. Since $u$ is superharmonic in $\R^N$, we have 
$u(x)\geq c|x|^{2-N}$ in $\R^N\setminus B_1$, for some constant $c>0$. Thus, $v$ is a solution of  \eqref{singpz} where $\phi=u^m$ satisfies \eqref{eqga} with $\gamma=m(N-2)\leq 2$. By Lemma \ref{zz}(i) it follows that \eqref{GM} has no positive solutions.

\medskip

\noindent (ii) From Lemma \ref{Glem0} we deduce 
\begin{equation}\label{ur1}
u(x)\geq c(N)\int_{\R^N}|x-y|^{2-N}\rho(y) dy\quad\mbox{ for all }x\in \R^N,
\end{equation}
where $C(N)>0$ is a constant. 

\noindent In particular, for $x=0$ we deduce $u(0)\geq c(N)\int_{\R^N}|y|^{2-N}\rho(y) dy=\infty$, contradiction.

\medskip

\noindent (iii) By \eqref{eqG02} in Lemma \ref{Glem0} and \eqref{ur1} we find 
$$
\begin{aligned}
v(x) &= c(N)\int_{\R^N} |x-y|^{2-N} \frac{u^m(y)}{ v^{s} (y)} dy\\
&\geq C\int_{\R^N} |x-y|^{2-N} u^m (y) dy\\
&\geq C\int_{\R^N} |x-y|^{2-N} \Big(\int_{\R^N} |y-z|^{2-N} \rho(z) dz \Big)^m dy.
\end{aligned}
$$
In particular, for $x=0$ we find
$$
\infty>v(0)\geq C \int_{\R^N} |y|^{2-N} \Big(\int_{\R^N} |y-z|^{2-N} \rho(z) dz \Big)^m dy=\infty,
$$
contradiction.

(iv) Observe that $u\in C^2(\R^N)$ is positive and satisfies $-\Delta u=h(x)  u^p$ in $\R^N$, where $h=v^{-q}$. In light of the second equation of \eqref{GM} one has 
$$
\Delta h=q(q+1)v^{-q-2}|\nabla v|^2+qv^{-q-1}(-\Delta v)\geq 0\quad\mbox{ in }\R^N.
$$
Also, since $v(x)\to 0$ as $|x|\to \infty$, one has $h\geq c>0$ in $\R^N$. Further, thanks to the assumption (iv) in Theorem \ref{thgs} one has $|\nabla h(x)|\leq \frac{C}{|x|}h(x)$ for $|x|>1$ large. 
By Lemma \ref{gspruck} it now follows that $u\equiv 0$, contradiction.

\medskip

\noindent{\bf Proof of Corollary \ref{cor1}.}  (i) Assume $\rho\equiv 0$, $\frac{N}{N-2}<p<\frac{N+2}{N-2}$ and that the system \eqref{GM0} has a positive solution $(u,v)$ with $u(x)\simeq |x|^{-a}$ and $0<a<\frac{(N-2)s+N}{m}$. Then, $v$ satisfies \eqref{singpz} with $\phi(x)=u^m(x)\simeq |x|^{-\gamma}$, $\gamma=am$ and $ 0<\gamma<(N-2)s+N$. By Lemma \ref{zz}(i) it follows that \eqref{GM} has no positive solutions if $am\leq 2$. Thus, $am>2$ and by Lemma \ref{zz}(ii) one has $v(x)\simeq |x|^{-b}$, where $b=\frac{am-2}{s+1}>0$. It follows that
$$
g(x):=\frac{u^m(y)}{v^s(y)}\simeq |x|^{-b-2}.
$$
Thus, there exist $c_2>c_1>0$ such that 
\begin{equation}\label{limg}
c_1|x|^{-b-2}\leq g(x) \leq c_2|x|^{-b-2} \quad\mbox{ for all }x\in \R^N\setminus B_1.
\end{equation}

We shall now check that $v$ satisfies \eqref{convr} in $\R^N\setminus B_2$ which further implies that \eqref{GM} has no positive solutions. 

By \eqref{eqG02} in Lemma \ref{Glem0} we deduce
\begin{equation}\label{limg1}
v(x) = c(N)\int_{\R^N} g(y)|x-y|^{2-N} dy \quad\mbox{ for all } x\in \R^N
\end{equation}
and then
\begin{equation}\label{limg2}
\nabla v(x) = -c(N)(N-2)\int_{\R^N} g(y) |x-y|^{-N}(x-y)  dy \quad\mbox{ for all } x\in \R^N.
\end{equation}
Let $x\in \R^N\setminus B_2$. We have 
\begin{equation}\label{limg3}
\begin{aligned}
|\nabla v(x)| &\leq C \int_{\R^N} g(y) |x-y|^{1-N} dy\\
&=C \left\{\int_{|y|<|x|/2}+\int_{|x|/2<|y|<2|x|}+\int_{|y|>2|x|}  \right\} g(y) |x-y|^{1-N} dy.
\end{aligned}
\end{equation}
To estimate the first integral in \eqref{limg3} we have for $|y|<|x|/2$ that $|x-y|\geq |x|-|y|>|x|/2$. Thus, by \eqref{limg1} we find
\begin{equation}\label{limg4}
\begin{aligned}
\int_{|y|<|x|/2}g(y) |x-y|^{1-N} dy & =\int_{|y|<|x|/2}g(y) |x-y|^{2-N} |x-y|^{-1}dy\\[0.1in]
&\leq \frac{2}{|x|} \int_{|y|<|x|/2}g(y) |x-y|^{2-N} dy\\
&\leq C \frac{v(x)}{|x|}.
\end{aligned}
\end{equation}
To estimate the second integral in \eqref{limg3} we have for $|x|/2<|y|<2|x|$ that $|x-y|\leq |x|+|y|<3|x|$. Hence, from \eqref{limg} and $v(x)\simeq |x|^{-b}$ we deduce

\begin{equation}\label{limg5}
\begin{aligned}
\int_{|x|/2<|y|<2|x|}g(y) |x-y|^{1-N} dy&\leq C \int_{|x|/2<|y|<2|x|} |y|^{-b-2} |x-y|^{1-N} dy\\
& \leq C|x|^{-b-2} \int_{|x-y|<3|x|} |x-y|^{1-N} dy\\
&\leq C|x|^{-b-1}\\
&\leq C \frac{v(x)}{|x|}.
\end{aligned}
\end{equation}
Finally, if $|y|>2|x|$ then $|x-y|\geq |y|-|x|\geq |y|/2$ we use \eqref{limg} to find 
\begin{equation}\label{limg6}
\begin{aligned}
\int_{|y|>2|x|}g(y) |x-y|^{1-N} dy&\leq C \int_{|y|>2|x|} |y|^{-b-2} |x-y|^{1-N} dy\\
&\leq C \int_{|y|>2|x|} |y|^{-b-1-N} dy\\
&= C|x|^{-b-1}\\
&\leq C \frac{v(x)}{|x|}.
\end{aligned}
\end{equation}
From \eqref{limg3}-\eqref{limg6} we deduce that $v$ satisfies \eqref{convr} and so the are no positive solutions of the system \eqref{GM}.
\medskip

(ii) Assume \eqref{GM} has a radial solution $(u,v)$. Then, letting $r=|x|$ we have that $v(r)$ satisfies
\begin{equation}\label{eqv}
\begin{cases}
-\big(r^{N-1}v'(r)\big)'=r^{N-1}g(r)\quad \mbox{ for all }r\geq 0,\\
\;\;\;  v(r)\to 0\quad\mbox{ as }\; r\to \infty,
\end{cases}
\end{equation}
where $g(r)=u^m(r)v^{-s}(r)>0$. From the above we find that $r\mapsto  r^{N-1}v'(r)$ is decreasing on $[0, \infty)$ and $v'(r)\leq 0$ for all $r\geq 0$. Integrating in \eqref{eqv} over $[0,r]$, for all $r>0$  one has
\begin{equation}\label{va1}
|v'(r)|=-v'(r)=r^{1-N}\int_0^r s^{N-1}g(s)ds.
\end{equation}
A new integration over $[r, \infty]$ yields
\begin{equation}\label{va2}
v(r)=\int_r^\infty t^{1-N}\Big(\int_0^t  s^{N-1}g(s)ds\Big) dt\quad\mbox{ for all }r>0.
\end{equation}
Then, from \eqref{va2} we find
$$
\begin{aligned}
v(r) & \geq \int_r^\infty t^{1-N}\Big(\int_0^r  s^{N-1}g(s)ds\Big) dt=\Big(\int_0^r  s^{N-1}g(s)ds\Big)\cdot \Big( \int_r^\infty t^{1-N} dt\Big)\\
& =\frac{r^{2-N}}{N-2}\int_0^r  s^{N-1}g(s)ds\quad\mbox{ for all }r>0.
\end{aligned}
$$
This last estimate combined with \eqref{va1} yields $v(r)\geq r|v'(r)|/(N-2)$ for all $r>0$. This shows that \eqref{convr} holds and thus, by Theorem \ref{thgs}(iv), the system \eqref{GM} has no radial solutions. 
\medskip

\section{Proof of Theorem \ref{thex} }

(i) By \eqref{roro} we have $\rho(x)\geq c|x|^{-a}$ for all $x\in \R^N\setminus B_1$. 
Thus, if $a\leq 2$ it follows that 
$$
\int_{\R^N\setminus B_1}\rho(x)|x|^{2-N} dx=\infty
$$ 
and then by Theorem \ref{thgs}(ii) we deduce that the system \eqref{GM} has no solutions. 

Assume now that $2<a\leq 2\big(1+\frac{1}{m}\big)$ holds. We will check that the integral condition in Theorem \ref{thgs}(iii) holds. 
Let $x\in \R^N\setminus B_1$. Then, 
$$
\begin{aligned}
\int_{\R^N} \frac{\rho(y)}{|x-y|^{N-2}}dy & \geq \int_{|y|>|x|} \frac{\rho(y)}{|x-y|^{N-2}}dy
\geq C \int_{|y|>|x|} \frac{|y|^{-a}}{|x-y|^{N-2}}dy\\
&\geq C\int_{|y|>|x|} \frac{|y|^{-a}}{(2|y|)^{N-2}}dy\geq C|x|^{2-a}.
\end{aligned}
$$
Then, since $a\leq 2\big(1+\frac{1}{m}\big)$ we find 
$$
\int_{\R^N}|x|^{2-N} \Big(\int_{\R^N} \frac{\rho(y)}{|x-y|^{N-2}}dy\Big)^m dx\geq 
C \int_{|x|>1}|x|^{2-N} \cdot |x|^{-m(a-2)}dx 
=\infty,
$$
and by Theorem \ref{thgs}(iii) we deduce that \eqref{GM} has no positive solutions. 

(ii)  Let 
\begin{equation}\label{bbb}
b=\frac{m(a-2)-2}{s+1}>0.
\end{equation}
Note that condition \eqref{eqaa} is equivalent to
\begin{equation}\label{eqab}
p(a-2)-bq\geq a.
\end{equation}
For $\overline M_1>\underline M_1>0$ and $\overline M_2>\underline M_2>0$, consider the set 
$$
\Sigma=\left\{u,v\in C(\R^N): 
\begin{aligned}
\underline M_1 Z_{a-2}& \leq u\leq  \overline M_1 Z_{a-2}\\
\underline M_2 Z_b & \leq v\leq  \overline M_2 Z_b
\end{aligned} \quad\mbox{ in }\R^N \right\},
$$
where $Z_\gamma (x)=(1+|x|^2)^{-\gamma/2}$, $\gamma>0$. 
For $(u,v)\in \Sigma$, let $Tu, Tv\in C^2(\R^N)$ be the unique solutions of 
\begin{equation}\label{eqz}
\begin{cases}
\displaystyle   -\Delta Tu=\frac{u^p}{v^q}+\rho(x) \quad\mbox{ in }\R^N\, , \\[0.1in]
\displaystyle   -\Delta Tv=\frac{u^m}{(Tv)^s} \quad\mbox{ in }\R^N,\\[0.1in]
Tu, Tv\to 0\mbox{ as }|x|\to \infty.
\end{cases}
\end{equation}
The existence and uniqueness of $Tv\in C^2(\R^N)$ follows from Lemma \ref{zz}(ii) with $\phi(x)=u^m(x)\simeq |x|^{-m(a-2)}$, $\gamma=m(a-2)$ and by \eqref{ams} one has  $2<\gamma=m(a-2)<(N-2)s+N$.

To prove the existence of $Tu$ we first note that from \eqref{roro}, \eqref{bbb}, \eqref{eqab}  and the definition of the set $\Sigma$ one has
\begin{equation}\label{za1}
\begin{aligned}
\frac{u^p}{v^q}+\rho(x) & \leq \overline M_1^p \underline M_2^{-q} Z_{a-2}^p Z_b^{-q}+\beta Z_a\\
&=\overline M_1^p \underline M_2^{-q} Z_{p(a-2)-bq} +\beta Z_a\\
&\leq \Big( \overline M_1^p \underline M_2^{-q} +\beta\Big) Z_a\quad\mbox{ in }\R^N.
\end{aligned}
\end{equation}
Since $a>2$, the function $f=\frac{u^p}{v^q}+\rho(x)$ satisfies \eqref{finiteint}. Indeed, by \eqref{za1} for all $x\in \R^N$ we have
$$
\begin{aligned}
\int_{\R^N} \mathcal{G}_0(x-y) f(y) dy&\leq C \int_{\R^N} |x-y|^{2-N}(1+|y|^2)^{-a/2}  dy\\[0.1in]
&\leq C\int_{|y|<2|x|} |x-y|^{2-N} dy+C\int_{|y|>2|x|} |x-y|^{2-N}|y|^{-a} dy\\[0.1in]
&\leq C\int_{|x-y|<3|x|} |x-y|^{2-N} dy+C\int_{|y|>2|x|} |y|^{2-N-a} dy\\[0.1in]
&\leq C(|x|^2+|x|^{2-a})<\infty.
\end{aligned}
$$
By Proposition \ref{prep0} it follows that 
$$
Tu=\int_{\R^N} \mathcal{G}_0(x-y)\Big[\frac{u^p(y)}{v^q(y)}+\rho(y)  \Big] dy,
$$
and this shows the existence and uniqueness of $Tu\in C^2(\R^N)$. 

Define $\Phi:\Sigma\to C(\R^N)\times C(\R^N)$ by $\Phi(u,v)=(Tu, Tv)$. Then, fixed points of $\Phi$ are solutions of \eqref{GM}. 

\begin{lemma}\label{MMz1}
If $\overline M_1>\underline M_1>0$ and $\overline M_2>\underline M_2>0$ satisfy 
\begin{subequations}
\begin{align}
(a-2)(N-a)\overline M_1&\geq \overline M_1^p \underline M_2^{-q}+\beta \, , \label{eqz1} \\[0.07in]
(a-2)N  \underline M_1 & =\alpha \, , \label{eqz2}\\
 b(N-b-2) \overline M_2&=\overline M_1^m \overline M_2^{-s} \, , \label{eqz3}\\ 
bN  \underline M_2 & = \underline M_1^m \underline M_2^{-s}\, , \label{eqz4}
\end{align}
\end{subequations}
then $\Phi(\Sigma)\subset \Sigma$. 
\end{lemma}
\begin{proof} Using  \eqref{za1} we have
$$
-\Delta Tu\leq \Big( \overline M_1^p \underline M_2^{-q} +\beta\Big) Z_a \quad\mbox{ in }\R^N.
$$
Now, by Lemma \ref{funcz}  the function 
$\displaystyle 
\overline \zeta=\frac{\overline M_1^p \underline M_2^{-q} +\beta}{(a-2)(N-a)} Z_{a-2}
$
satisfies
$$
\begin{cases}
-\Delta \overline \zeta\geq \Big( \overline M_1^p \underline M_2^{-q} +\beta\Big) Z_a\geq -\Delta Tu \quad\mbox{ in }\R^N\, ,\\
\;\;\; \overline \zeta(x)\,, \, Tu(x)\to 0\quad\mbox{ as }|x|\to \infty.
\end{cases}
$$
By the maximum principle and \eqref{eqz1} we find $Tu\leq  \overline\zeta \leq \overline M_1 Z_{a-2}$ in $\R^N$. Next,
$$
-\Delta Tu\geq \rho(x) \geq \alpha Z_a  \quad\mbox{ in }\R^N.
$$
Thus, by Lemma \ref{funcz} and \eqref{roro} the function $\displaystyle \underline \zeta=\frac{\alpha}{(a-2)N } Z_{a-2}$ satisfies 
$$
\begin{cases}
-\Delta  \underline \zeta   \leq \alpha Z_{a}\leq \rho(x) \leq -\Delta Tu  \quad\mbox{ in }\R^N\, ,\\
\;\;\;  \underline \zeta (x)\, , Tu(x)\to 0\quad\mbox{ as }|x|\to \infty.
\end{cases}
$$
By the maximum principle and \eqref{eqz2} one has $Tu\geq \underline \zeta=\underline M_1 Z_{a-2}$ in $\R^N$. 

\medskip

It remains to establish the inequality $\underline M_2 Z_b\leq Tv\leq \overline M_2 Z_b$ in $\R^N$. 

From  $\underline M_1 Z_{a-2}\leq u\leq \overline M_1 Z_{a-2}$ in $\R^N$, one has that $Tv$ satisfies 
$$
\underline M_1^m Z_{a-2}^m (Tv)^{-s} \leq -\Delta Tv\leq \overline M_1^m Z_{a-2}^m (Tv)^{-s}\quad\mbox{ in }\R^N.
$$
We observe that from \eqref{bbb}, \eqref{eqz3} and Lemma \ref{funcz}, the function  $\overline \xi=\overline M_2 Z_b$ satisfies 
$$
-\Delta \overline \xi \geq b(N-b-2)  \overline M_2 Z_{b+2}= \overline M_1^m Z_{a-2}^m \overline \xi^{-s}\quad \mbox{ in }\R^N.
$$
Hence
$$
\begin{cases}
-\Delta Tv\leq \overline M_1^m Z_{a-2}^m (Tv)^{-s}   \quad\mbox{ in }\R^N\, ,\\[0.07in]
-\Delta \overline \xi \geq  \overline M_1^m Z_{a-2}^m \overline \xi^{-s}\quad \mbox{ in }\R^N\, , \\[0.05in]
\;\;\; T v (x)\, , \,  \overline \xi (x)\to 0\quad\mbox{ as }|x|\to \infty.
\end{cases}
$$
With the same method as in the proof of the uniqueness in Lemma \ref{eax} we deduce $Tv \leq \overline \xi=\overline M_2 Z_b$ in $\R^N$.  Finally, from  \eqref{bbb}, \eqref{eqz4} and Lemma \ref{funcz} we deduce that $\underline \xi=\underline M_2 Z_b$ satisfies 
$$
-\Delta \underline \xi \leq bN \underline M_2 Z_{b+2}= \underline M_1^m Z_{a-2}^m \underline \xi^{-s}\quad \mbox{ in }\R^N.
$$
As above, it follows that $Tv\geq \underline \xi=\underline M_2 Z_b$ in $\R^N$ and this concludes our proof.
\end{proof}

\begin{lemma}\label{conste}
There exist $\varepsilon, \delta>0$ such that if  $\alpha, \beta>0$ satisfy \eqref{ed}, then, there exist  $\overline M_1>\underline M_1>0$ and $\overline M_2>\underline M_2>0$ that fulfill \eqref{eqz1}-\eqref{eqz4} and thus $\Sigma$ is an invariant set of $\Phi$.
\end{lemma}
\begin{proof}
It is enough to show the existence of $\overline M_1>\underline M_1>0$ and $\overline M_2>\underline M_2>0$ that satisfy 
\begin{subequations}
\begin{align}
\frac{(a-2)(N-a)}{2} \overline M_1&= \overline M_1^p \underline M_2^{-q}\quad\mbox{ and } \quad \frac{(a-2)(N-a)}{2} \overline M_1\geq \beta \, , \label{eqz1a} \\[0.07in]
(a-2)N  \underline M_1 & =\alpha \, , \label{eqz2a}\\
 b(N-b-2) \overline M_2&=\overline M_1^m \overline M_2^{-s} \, , \label{eqz3a}\\ 
bN   \underline M_2 & = \underline M_1^m \underline M_2^{-s}\, , \label{eqz4a}
\end{align}
\end{subequations}
Starting from \eqref{eqz2a}, then \eqref{eqz4a}, \eqref{eqz1a}$_1$ and then \eqref{eqz3a} we find
\begin{equation}\label{MMM}
\begin{aligned}
\underline M_1&=A\alpha,\\[0.05in]
\overline M_1&=C\alpha^\sigma,
\end{aligned} 
\qquad 
\begin{aligned}
\underline M_2&=B\alpha^{\frac{m}{s+1}},\\[0.05in]
\overline M_2&=D\alpha^{\frac{\sigma m}{s+1}},
\end{aligned} 
\end{equation}
where
\begin{equation}\label{rangm}
\begin{aligned}
A&=\frac{1}{(a-2)N}\, ,\\
B&=\left(\frac{A^m}{bN}  \right)^{\frac{1}{s+1}},
\end{aligned} 
\qquad 
\begin{aligned}
C&=\left(\frac{(a-2)(N-a) B^q}{2}  \right)^{\frac{1}{p-1}},\\
D&=\left(\frac{C^m}{b(N-b-2)}  \right)^{\frac{1}{s+1}}.
\end{aligned} 
\end{equation}
It remains to determine the rage of $\alpha$ and $\beta$ such that 
\begin{equation}\label{rangab}
\overline M_1>\underline M_1>0 \, , \quad \overline M_2>\underline M_2>0\quad\mbox{ and }\; \quad \frac{(a-2)(N-a)}{2} \overline M_1\geq \beta.
\end{equation}
In light of \eqref{MMM} and \eqref{rangm} the first two inequalities in \eqref{rangab} hold if 
\begin{equation}\label{range1}
0<\alpha<\min \left\{ \Big(\frac{C}{A}\Big)^{\frac{1}{1-\sigma}}, \Big(\frac{D}{B}\Big)^{\frac{s+1}{m(1-\sigma)}}\right\}.
\end{equation}
Also, the third inequality in \eqref{rangab} is equivalent to 
\begin{equation}\label{range2}
\alpha<\beta<\frac{(a-2)(N-a)}{2} C\alpha^\sigma,
\end{equation}
which also implies 
\begin{equation}\label{range3}
0<\alpha<\left[\frac{(a-2)(N-a)}{2} C  \right]^{\frac{1}{1-\sigma}}
\end{equation}
Let now 
$$
\varepsilon =\min \left\{ \Big(\frac{C}{A}\Big)^{\frac{1}{1-\sigma}}, \Big(\frac{D}{B}\Big)^{\frac{s+1}{m(1-\sigma)}}, \left[\frac{(a-2)(N-a)}{2} C   \right]^{\frac{1}{1-\sigma}}\right\}
$$
and 
$$
\delta= \frac{(a-2)(N-a)}{2} C.
$$
Then, if \eqref{ed} is fulfilled, then \eqref{range1}-\eqref{range3} hold, which imply that \eqref{eqz1a}-\eqref{eqz4a}  and \eqref{rangab} hold. By Lemma \ref{MMz1} it now follows that $\Sigma$ is an invariant set of $\Phi$ and this concludes the proof.
\end{proof}

To finish the proof of Theorem \ref{thex}(ii) we follow the same approach as in the proof of Theorem \ref{th1}(iii). We use the Schauder fixed point theorem on $C(\overline{B_n})\times C(\overline{B_n})$ and then, a diagonal process allows us to deduce the existence of a positive solution to \eqref{GM}.
\qed

\medskip

\noindent{\bf Proof of Corollary \ref{corex}.}  Since $a<N$ and $m\leq s+1$ it follows that \eqref{ams} holds. Also, from $2\Big(1+\frac{1}{m}\Big)<a$, $m\leq p-1$ and $0<\sigma<1$ we find
$$
\frac{2p}{p-1}\leq 2\Big(1+\frac{1}{m}\Big)\leq a+\sigma \Big\{2\Big(1+\frac{1}{m}\Big)-a\Big\},
$$
so \eqref{eqaa} holds. By Theorem \ref{thex} we deduce the existence of a positive solution to \eqref{GM}. 
\qed

\medskip

\noindent{\bf Proof of Corollary \ref{cor2}.}  Part (i) follows from Theorem \ref{thgs}(i). 

(ii) Assume $p>\frac{N}{N-2}$, so that $\frac{2p}{p-1}<N$. We show that one can find $q,m,s$ and $a$ that satisfy $2\big(1+\frac{1}{m}\big)<a<N$, $0<\sigma<1$, \eqref{ams} and \eqref{eqaa}. 

Let $a>0$ be close to $N$ such that $2<\frac{2p}{p-1}<a<N$. By letting $m>0$ large enough, one has $2\big(1+\frac{1}{m}\big)<a<N$. At this point $a$ is fixed. Now, by taking $s>m$ large we have that \eqref{ams} is fulfilled.  Next, by letting $q>0$ small one has $\sigma\in (0,1)$ is small and fulfils \eqref{eqaa} in Theorem \ref{thex}. It remains to choose $\beta>\alpha>0$ in \eqref{roro} and \eqref{ed} in order to apply Theorem \ref{thex}(ii) and thus to deduce the existence of a positive solution to \eqref{GM}. 
\qed

\medskip

\noindent{\bf Proof of Corollary \ref{cor3}.}  Part (i) follows from Theorem \ref{thgs}(i) and Corollary \ref{cor1}(ii). 

(ii) Assume $p>\frac{N+2}{N-2}$. It is well known (see \cite[Appendix A]{GS81}) that the function 
$$
w(x)=\left[\frac{A \sqrt{N(N-2)}}{A^2+|x|^2}\right]^{\frac{N-2}{2}}\,,\quad A>0,
$$ satisfies
$$
-\Delta w=w^{\frac{N+2}{N-2}}\quad\mbox{ in }\R^N.
$$
Take now 
$$
q=p-\frac{N+2}{N-2}\,,\quad s>0\,\mbox{ \;\; and \;\; }\; m=\frac{N+2}{N-2}+s.
$$
Then $u=v=w$ satisfy \eqref{GM}.
\qed

\end{document}